\documentclass{article}

\usepackage[utf8]{inputenc}
\usepackage[T1]{fontenc}
\usepackage{lmodern}
\usepackage[a4paper]{geometry}
\usepackage{graphicx}
\usepackage{amsmath}
\usepackage{amsfonts}
\usepackage{amssymb}
\usepackage{amsthm}
\usepackage{mathtools}
\usepackage{stmaryrd}
\usepackage{fontawesome}
\usepackage[french, english]{babel}
\usepackage{subcaption}
\usepackage[hidelinks]{hyperref}

\newtheorem{thm}{Theorem}
\newtheorem{pro}[thm]{Proposition}
\newtheorem{lem}[thm]{Lemma}
\newtheorem{df}[thm]{Definition}
\newtheorem{prb}[thm]{Problem}
\newtheorem{cor}[thm]{Corollary}
\newtheorem{exe}[thm]{Example}
\newtheorem{rem}[thm]{Remark}
\newtheorem{nt}[thm]{Notation}

\newtheorem{frthm}{Théorème}
\newtheorem{frdf}[frthm]{Définition}
\newtheorem{frpro}[frthm]{Proposition}
\newtheorem{frlem}[frthm]{Lemme}

\usepackage{scalerel}
\DeclareMathOperator*{\connectsum}{\scalerel*{\#}{\sum}}

\newcommand{\intent}[1]{\llbracket#1\rrbracket}
\usepackage{tikz}
\usetikzlibrary{arrows, decorations}

\tikzstyle{mini diagram line} = [line width=0.75]
\tikzstyle{diagram line} = [line width=2]
\tikzstyle{diagram arrow} = [diagram line,->,>=stealth]

\tikzset{
    undercrossing/.style args={at #1}{
        decorate,
        decoration={
            moveto,
            pre=curveto, pre length=(#1*\pgfmetadecoratedpathlength-2.5*\pgflinewidth),
            post=curveto, post length=((1-#1)*\pgfmetadecoratedpathlength-2.5*\pgflinewidth)
        }
    },
    undercrossing/.default={at 0.5}
}

\makeatletter
\newcommand{\todo}{
    \@ifstar
        \@todostar
        \@todonostar
}
\newcommand{\@todostar}[1]{\PackageWarning{TODO}{#1}}
\newcommand{\@todonostar}[1]{\@todostar{#1}\emph{\color{red}TODO: #1}}
\makeatother

\title{On the modular Jones polynomial}
\author{Guillaume Pagel}

\begin{document}

\maketitle

\begin{abstract}
A major problem in knot theory is to decide whether the Jones polynomial detects the unknot. In this paper we study a weaker related problem, namely whether the Jones polynomial reduced modulo an integer $n$ detects the unknot. The answer is known to be negative for $n=2^k$ with $k\geq 1$ and $n=3$. Here we show that if the answer is negative for some $n$, then it is negative for $n^k$ with any $k\geq 1$. In particular, for any $k\geq 1$, we construct nontrivial knots whose Jones polynomial is trivial modulo~$3^k$.

\textbf{Keywords:} Knot, Jones polynomial, Kauffman bracket, $n$-trivial knot, connected sum, Legendre formula, modular arithmetic.
\end{abstract}

\selectlanguage{french}
\begin{abstract}
Un problème majeur en théorie des noeuds est de décider si le polynôme de Jones détecte le noeud trivial. Dans cet article nous étudions une question similaire plus faible, c'est-à-dire si le polynôme de Jones réduit modulo un entier $n$ détecte le noeud trivial. On sait que la réponse est négative pour $n=2^k$ et $n=3$. On montre ici que si cette affirmation est fausse pour un entier $n$, alors elle l'est aussi pour $n^k$ avec $k\geq 1$. En particulier, on construit des noeuds non-triviaux avec un polynôme de Jones trivial modulo $3^k$.

\textbf{Mots-clés :} Noeud, Polynôme de Jones, crochet de Kauffman, noeud $n$-trivial, somme connexe, formule de Legendre, Arithmétique modulaire.
\end{abstract}
\selectlanguage{english}

\selectlanguage{french}
\section*{Version française abrégée}
\subsection*{Petite introduction}
L'un des problèmes majeurs de la théorie des noeuds est de développer des méthodes pour déterminer le plus simplement possible si un noeud donné est isotope au noeud trivial ou non. L'une de ces méthodes est l'utilisation d'un invariant, l'un des plus connus étant le polynôme de Jones. Une question encore ouverte à l'heure actuelle est de savoir si celui-ci \emph{détecte} le noeud trivial, c'est-à-dire que seul le noeud trivial ait un polynôme de Jones égal à $1$.

Dans cet article, on propose d'étudier un problème proche, à savoir si il existe des noeuds non-triviaux dont le polynôme de Jones est trivial modulo un entier $n$, que l'on appellera noeuds $n$-triviaux par la suite.
\begin{frdf}[Noeud $n$-trivial]
On dit qu'un noeud non-trivial $K$ est \emph{$n$-trivial} si son polynôme de Jones $V(K)$ vérifie $V(K)\equiv1[n]$.
\end{frdf}

Dans le papier de S.~Eliahou et J.~Fromentin \cite{eliahou2017remarkable}, on a une construction de noeuds premiers\footnote{On dit d'un noeud qu'il est \emph{premier} si il n'est pas trivial et si l'on ne peut pas l'écrire comme somme connexe de deux noeuds non-triviaux.} $n$-triviaux pour $n$ un entier s'écrivant comme une puissance de $2$, ainsi que l'existence de noeuds $3$-triviaux, mais pas d'informations quant aux autres entiers. Afin d'apporter quelques réponses à cette question, on propose de montrer le théorème suivant~:
\begin{frthm}\label{frthm:nktrivial}
Si il existe un noeud $n$-trivial pour un certain $n\geq 2$, alors quel que soit $k\geq 1$ il existe des noeuds non-premiers $n^k$-triviaux.
\end{frthm}
Dans la sous-section suivante, on donne les étapes clés de la démonstration constructive de ce théorème.

\subsection*{Résumé de la preuve}
Pour arriver à ce résultat, on aura besoin de quelques propriétés du polynôme de Jones ainsi que des coefficients binômiaux. On commence par rappeler la définition d'un noeud~:
\begin{frdf}[Noeud]
On appelle \emph{noeud} l'image du cercle $S^1$ par un plongement dans $\mathbb{R}^3$ à déformation près. Le \emph{noeud trivial} est donné par le plongement canonique.
\end{frdf}

On peut voir une représentation d'un noeud non-trivial dans la figure~\ref{fig:knot}, que nous appellerons $\gamma$ par la suite. Cette définition peut être généralisée au plongement de plusieurs cercles, ce qui donnera un entrelac.

Lorsque l'on effectue une projection du noeud sur un plan, on crée un \emph{diagramme} du noeud. A partir de celui-ci on peut calculer son polynôme de Jones via le \emph{crochet de Kauffman} suivant le \emph{modèle des états} introduit par L.H.~Kauffman dans \cite{kauffman1990state}. Pour résumer rapidement, il s'agit de couper chaque croisement du noeud selon deux possibilités, ce qui nous donne deux nouveaux diagrammes pondérés chacun par un coefficient, appelé \emph{état} du noeud. On peut résumer formellement ce modèle par trois règles (Voir~(\ref{eqn:KauffmanBracketRules}) en section~\ref{sec:jonespolynomial}).
\begin{frdf}[Polynôme de Jones]
Pour $K$ un noeud orienté, on définit le \emph{polynôme de Jones} $V(K)$ appartenant à $\Lambda=\mathbb{Z}[t^{-1}, t]$ en utilisant le crochet de Kauffman comme~:
\[
V(K) = \left((-\tau^3)^{-w(K)}\langle K \rangle\right)_{\tau=t^{-\frac{1}{4}}}
\]
où $\langle\cdot\rangle$ désigne le crochet de Kauffman et $w(K)$ l'entortillement de $K$, défini comme la différence entre le nombre de croisements positifs et négatifs (voir les figures~\ref{fig:positivecrossing} et~\ref{fig:negativecrossing}).
\end{frdf}

Il faut préciser qu'en général le polynôme de Jones vit dans $\mathbb{Z}\left[\sqrt{t},\sqrt{t^{-1}}\right]$, mais dans le cas des noeuds nous pouvons considérer $\Lambda$ à la place \cite[Theorem~2]{jones1985polynomial}. Cet invariant possède beaucoup de propriétés, ainsi le calcul du polynôme de Jones de la somme connexe de deux noeuds revient à une multiplication~\cite[Theorem 6]{jones1985polynomial}~:
\begin{frpro}\label{frpro:sommeconnexe}
Pour $K_1$ et $K_2$ deux noeuds ayant comme polynôme de Jones $V_1$ et $V_2$ respectivement, le polynôme de Jones de la somme connexe $K_1\# K_2$ est $V_1V_2$.
\end{frpro}

On utilisera la notation $\#$ pour la somme connexe, et on écrirera $\#(K,n)$ pour la somme connexe de $n$ fois le noeud $K$. Cette propriété est une des deux clés pour montrer le théorème~\ref{frthm:nktrivial}. On a à présent besoin de propriétés d'arithmétique modulaire.
\begin{frpro}\label{frpro:binome}
Pour $n,k\geq2$ et $i\in\intent{1,k-1}$, la puissance $n^{k-i}$ divise $\binom{n^{k-1}}{i}$.
\end{frpro}

La preuve se base d'une part sur la formule de Legendre \cite[theorem~1.2.3 p.~6]{boros2004irresistible} (voir \cite[XVI p.~8]{legendre1808essai} pour l'original) et sur l'étude des $p$-valuations des coefficients binomiaux. Cette propriété donne alors le lemme suivant~:
\begin{frlem}\label{frlem:puissancepolymodulaire}
Soient $P$ et $Q$ deux polynômes à coefficients entiers tels que l'on ait $P=1+nQ$ pour un certain $n\in\mathbb{N}$. Alors pour tout $k\geq 1$ on a $P^{n^{k-1}}\equiv 1 \left[n^k\right]$.
\end{frlem}

La preuve du théorème~\ref{frthm:nktrivial} devient alors une formalité, il suffit de combiner la propriété~\ref{frpro:sommeconnexe} et le lemme~\ref{frlem:puissancepolymodulaire} pour obtenir le résultat souhaité.

\subsection*{Conséquence et poursuites}

La conséquence principale de ce théorème est l'existence quel que soit $k$ de noeuds non-premiers $2^k$-triviaux et $3^k$-triviaux, ce qui vient compléter les premières découvertes de S.~Eliahou et J.~Fromentin dans \cite{eliahou2017remarkable}. On peut voir en exemple le noeud $\#(\gamma,3)$ qui est $9$-trivial (figure~\ref{fig:knot9trivial}).

Cependant on n'a aucune information sur d'autres modules. Découvrir un noeud $6$-trivial serait en particulier très intéressant, étant à la fois $2$-trivial et $3$-trivial il pourrait permettre de déterminer si la propriété $n$-trivial est multiplicative par rapport à $n$.

Cette même propriété est définie sur le polynôme de Jones, on peut imaginer une définition similaire sur le crochet de Kauffman, donnant des résultats certainement plus forts et en lien avec ceux de cet article. Cette approche a déjà été utilisée dans \cite{eliahou2017remarkable}, mais seulement pour des tangles algébriques.

Il reste également à déterminer le nombre minimal de croisements nécessaires pour qu'un noeud puisse être $n$-trivial, et de constater si ces noeuds $n$-triviaux \og{}minimaux\fg{} sont premiers.


\selectlanguage{english}

\section{Introduction}

One of the major aims of knot theory is to determine as simply as possible whether a given knot is isotopic to the unknot. The Jones polynomial is a knot invariant living in the ring of Laurent polynomials over the integers. A long-standing question is to determine whether the Jones polynomial can \emph{detect} the unknot, meaning that the unknot is the only knot with Jones polynomial equal to $1$. In case of links we know that this invariant does not detect the unlink with at least two components: this was proved first by M.~Thistlethwaite \cite{thistlethwaite2001links} for links with $2$ and $3$ components, then generalised by S.~Eliahou, L.H.~Kauffman and M.~Thistlethwaite \cite{eliahou2003infinite}, but leave unanswered the case of knots. However we know how to construct mutant knots that aren't distinguished by the Jones polynomial \cite{kohno_jones_1990}.

The idea here is to study the Jones polynomial in a modular way, in order to better understand structures formed by knots. Also thanks to the modulo operation, some of the coefficients of the polynomial will disappear, and sometimes the Jones polynomial modulo an integer $n$ will become trivial. Nontrivial knots with this property will be called $n$-trivial.
\begin{df}[$n$-trivial knot]
We say that a nontrivial knot $K$ is \emph{$n$-trivial} if its Jones polynomial $V(K)$ satisfy $V(K)\equiv1[n]$.
\end{df}

A modular version of the Jones polynomial problem is then:
\begin{prb}\label{prb:existntrivial}
Given any integer $n\geq2$, do there exist $n$-trivial knots?
\end{prb}

For all $k\geq 1$, the existence of $2^k$-trivial knots has been established by S.~Eliahou and J.~Fromentin in \cite{eliahou2017remarkable}. They also mention the existence of $3$-trivial knots. Essentially nothing else is known, except that there are no $5$-trivial knots up to 16 crossings.

In this paper we provide some answers to this problem, in fact we will claim the following theorem:
\begin{thm}\label{thm:nktrivial}
If there exists an $n$-trivial knot for some $n\geq 2$, then for all $k\geq 1$ there exists an $n^k$-trivial knot.
\end{thm}

This result allows us to give a positive answer to problem~\ref{prb:existntrivial} for integers of the form $3^k$, and gives a new proof for $n=2^k$. Furthermore, the proof of theorem~\ref{thm:nktrivial} is constructive and gives an explicit way to obtain these knots. Let us remark here that the knots constructed in \cite{eliahou2017remarkable} are prime, which is not the case in this paper. To achieve that, we will need some results in modular arithmetic.
Here is the structure of this paper: in section~\ref{sec:jonespolynomial}, we recall some properties of the Jones polynomial, then in section~\ref{sec:modularjonespolynomial} we give a proof of theorem~\ref{thm:nktrivial}. In section~\ref{sec:arithmeticproperties}, we prove some arithmetic results we used in the previous section. Finally in section~\ref{sec:furtherwork} we finish with some open questions.
\section{Knots and the Jones polynomial}\label{sec:jonespolynomial}

We start this section with a formal definition of a \emph{knot}:
\begin{df}[Knot]
A \emph{knot} is the image of an embedding of the circle $S^1$ into $\mathbb{R}^3$ up to deformation. The \emph{unknot} is given by the canonical embedding. 
\end{df}

\begin{figure}[h]
    \centering
    \includegraphics{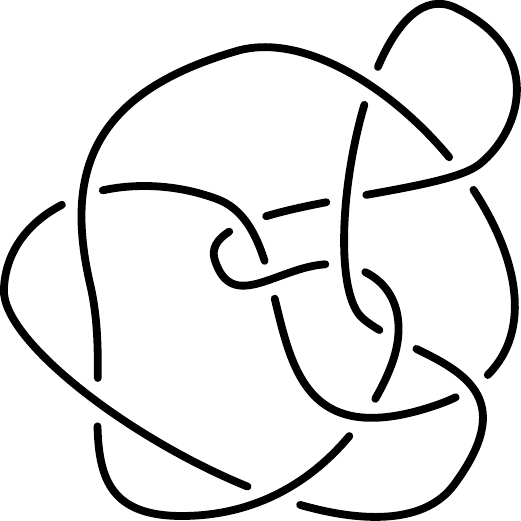}
    \caption{A representation of $\gamma$, a $3$-trivial prime knot with 12 crossings \cite[knot 12n659]{cha2011knotinfo}.\\ \selectlanguage{french}\emph{Une représentation de $\gamma$, un noeud $3$-trivial premier avec 12 croisements \cite[knot 12n659]{cha2011knotinfo}.}\selectlanguage{english}}
    \label{fig:knot}
\end{figure}
An example of a nontrivial knot can be seen in figure~\ref{fig:knot}. We can consider the embedding of a disjoint union of several circles, in that case we will obtain another object called \emph{link}. The number of circles embedded gives the number of \emph{components} of the link, knots are particular links with one component.

The Jones polynomial discovered by V.F.R. Jones \cite{jones1985polynomial} can be contructed via the \emph{states model} introduced by L.H. Kauffman \cite{kauffman1990state}. After we project the knot on a plane, the main idea is to split each crossing recursively. This operation create new \emph{states} weighted with a coefficient on each step. More precisely, the \emph{Kauffman bracket} follows the rules:
\begin{equation}
\left\{
\begin{array}{rl}\label{eqn:KauffmanBracketRules}\tag{$\ast$}
    \left\langle\begin{tikzpicture}[baseline={(0,0.1)}, scale=0.4]
    \draw[mini diagram line] (0.5,0.5) circle[radius=0.5];
\end{tikzpicture}\right\rangle &=  1,\\
    \left\langle K\sqcup\begin{tikzpicture}[baseline={(0,0.1)}, scale=0.4]
    \draw[mini diagram line] (0.5,0.5) circle[radius=0.5];
\end{tikzpicture}\right\rangle &= -(\tau^{-2}+\tau^2)\langle K\rangle,\\
    \left\langle\begin{tikzpicture}[baseline={(0,0.1)}, scale=0.4]
    \draw[mini diagram line] (0,0) -- (1,1);
    \draw[mini diagram line, undercrossing] (1,0) -- (0,1);
\end{tikzpicture}\right\rangle &= \tau\left\langle\begin{tikzpicture}[baseline={(0,0.1)}, scale=0.4]
    \draw[mini diagram line] (0,0) .. controls (0.5,0.5) .. (0,1);
    \draw[mini diagram line] (1,0) .. controls (0.5,0.5) .. (1,1);
\end{tikzpicture}\right\rangle + \tau^{-1}\left\langle\begin{tikzpicture}[baseline={(0,0.1)}, scale=0.4]
    \draw[mini diagram line] (0,0) .. controls (0.5,0.5) .. (1,0);
    \draw[mini diagram line] (0,1) .. controls (0.5,0.5) .. (1,1);
\end{tikzpicture}\right\rangle.
\end{array}\right.
\end{equation}
The second rule means that if we have a diagram isotopic to a circle next to an other diagram $K$ without any crossing between them, then we can replace this circle by a coefficient. The third rule explains how to split each crossing locally, and the first one treats the case of the unknot. Once the bracket is computed, the only thing left to do is to normalize and change the variable.
\begin{df}[Jones polynomial]
For an oriented knot $K$, we can construct the \emph{Jones polynomial} $V(K)$ living in $\Lambda=\mathbb{Z}[t,t^{-1}]$ using the Kauffman bracket as:
\[
V(K) = \left((-\tau^3)^{-w(K)}\langle K \rangle\right)_{\tau=t^{-\frac{1}{4}}}
\]
where $\langle\cdot\rangle$ denote the Kauffman bracket and $w(K)$ is the \emph{writhe} of $K$, defined as the difference between the number of positive and negative crossings (see figure~\ref{fig:crossingsign}).
\end{df}

We have to specify that the Jones polynomial naturally lives in $\mathbb{Z}\left[\sqrt{t}, \sqrt{t^{-1}}\right]$, but in the case of knots we can consider the domain $\Lambda$ instead \cite[Theorem~2]{jones1985polynomial}.
\begin{figure}[h]
    \centering
    \begin{subfigure}{0.45\textwidth}
        \centering
        \begin{tikzpicture}
    \draw[diagram arrow] (0,0) -- (1,1);
    \draw[diagram line, undercrossing={at 0.66666}] (1,0) -- (0.25,0.75);
    \draw[diagram arrow] (0.25,0.75) -- (0,1);
\end{tikzpicture}
        \caption{Positive crossing\\ \selectlanguage{french}\emph{Croisement positif}\selectlanguage{english}}
        \label{fig:positivecrossing}
    \end{subfigure}
    \begin{subfigure}{0.45\textwidth}
        \centering
        \begin{tikzpicture}
    \draw[diagram arrow] (1,0) -- (0,1);
    \draw[diagram line, undercrossing={at 0.66666}] (0,0) -- (0.75,0.75);
    \draw[diagram arrow] (0.75,0.75) -- (1,1);
\end{tikzpicture}
        \caption{Negative crossing\\ \selectlanguage{french}\emph{Croisement négatif}\selectlanguage{english}}
        \label{fig:negativecrossing}
    \end{subfigure}
    \caption{Each crossing of an oriented knot can be identified to one of figure~\ref{fig:positivecrossing} or~\ref{fig:negativecrossing}.\\ \selectlanguage{french}\emph{Tout croisement d'un noeud orienté peut être identifié à l'une des figures~\ref{fig:positivecrossing} ou~\ref{fig:negativecrossing}.}\selectlanguage{english}}
    \label{fig:crossingsign}
\end{figure}
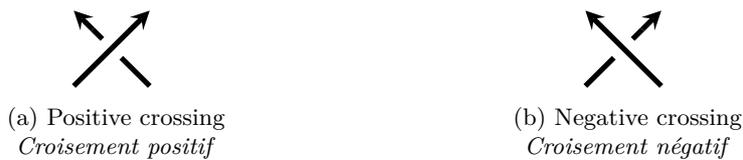

The computation by this method is simple but involves $2^l$ terms, with $l$ the number of crossings.
\begin{exe}\label{exe:3trivial}
The Jones polynomial of the knot $\gamma$ depictured in figure~\ref{fig:knot} is:
\[
V(\gamma) = 1-3t+6t^2-9t^3+12t^4-12t^5+12t^6-9t^7+6t^8-3t^9.
\]
We observe that this knot is $3$-trivial.
\end{exe}

The Jones polynomial has very interesting properties, we recall here a well known one~\cite[Theorem 6]{jones1985polynomial}:
\begin{pro}\label{pro:jonesconnectedsum}
For two knots $K_1$ and $K_2$, the Jones polynomial of the connected sum of $K_1$ and $K_2$ is $V_1V_2$ where $V_1$ and $V_2$ are the Jones polynomial of $K_1$ and $K_2$ respectively.
\end{pro}

The connected sum is the same as the topological one, i.e. we cut each knot in one point and glue the endpoints created on one knot to the other one without crossing. The end result is independent of the cutting points chosen. In the sequel, this operation will be denote by $\#$, and the connected sum of $n$ times the knot $K$ will be denoted as $\#(K,n)$.

As a consequence of proposition~\ref{pro:jonesconnectedsum}, the existence of $p$-trivial knots with $p$ prime leads to the existence of $p$-trivial \emph{prime} knots. We define this property below:
\begin{df}[Prime knot]\label{df:primeknot}
A knot is \emph{prime} if it is not trivial and if it cannot be written as a connected sum of two non-trivial knots.
\end{df}

For example, the knot $\gamma$ represented in figure~\ref{fig:knot} is prime. Now, we can properly state:
\begin{pro}
For $p$ a prime number, if there exist a $p$-trivial knot, then there exist a $p$-trivial prime knot.
\end{pro}
\begin{proof}
Suppose that there exists a $p$-trivial non-prime knot $K$. We may assume $K=K_1\#K_2$ with $K_1$, $K_2$ both non-trivial knots and $K_1$ prime. As $K$ is $p$-trivial, its Jones polynomial in the ring $\mathcal{R}=\mathbb{Z}/p\mathbb{Z}\left[\sqrt{t}, \sqrt{t}^{-1}\right]$ is $V\left(K\right)=_{\mathcal{R}}1$. By proposition~\ref{pro:jonesconnectedsum}, we have $V\left(K_1\right)V\left(K_2\right)=_{\mathcal{R}}1$. As $p$ is prime, the ring $\mathcal{R}$ is an integral domain, so we deduce that $V\left(K_1\right)$ and $V\left(K_2\right)$ are constant over $\mathcal{R}$. By~\cite[Theorem~15]{jones1985polynomial}, we know that the Jones polynomial of any knot evaluated at $t=1$ is $1$. It follows here that $V\left(K_1\right)=_{\mathcal{R}}V\left(K_2\right)=_{\mathcal{R}}1$. Hence $K_1$ is a $p$-trivial prime knot, as desired.
\end{proof}
\section{Proof and consequences of theorem~\ref{thm:nktrivial}}\label{sec:modularjonespolynomial}

In this section we prove theorem~\ref{thm:nktrivial} and study its consequences. We use a lemma we will prove in the next section.
\begin{lem}\label{lem:powerpolynomialmodular}
Let $P$ and $Q$ be two polynomials over the integers such that $P = 1 +nQ$ for some $n\geq2$. Then we have:
\begin{align*}
    P &\equiv 1 [n]\\
    P^{n^{k-1}} &\equiv 1 \left[n^k\right]
\end{align*}
\end{lem}

We can now prove theorem~\ref{thm:nktrivial} by construction:
\begin{proof}[Proof of theorem~\ref{thm:nktrivial}]
Let $K$ be a $n$-trivial knot. We denote by $V(K) = 1+nP$ the Jones polynomial of $K$. Hence the connected sum of $n^{k-1}$ times the knot $K$ will be:
\[
V\left(\#\left(K,n^{k-1}\right)\right) = (V(K))^{n^{k-1}} = (1+nP)^{n^{k-1}}
\]
according to proposition~\ref{pro:jonesconnectedsum}. By lemma~\ref{lem:powerpolynomialmodular}, the knot $\#\left(K,n^{k-1}\right)$ is $n^k$-trivial.
\end{proof}

As an immediate consequence we have the following result:
\begin{cor}
For all integers $k\geq 2$ there exist $3^k$-trivial and $2^k$-trivial non-prime knots.
\end{cor}
\begin{exe}
As we saw in example~\ref{exe:3trivial}, the knot $\gamma$ in figure~\ref{fig:knot} is $3$-trivial. We can construct a $9$-trivial knot in the form of $\#(\gamma,3)$ represented in figure~\ref{fig:knot9trivial}. Its Jones polynomial is:
\[
-27t^{27}+162t^{26}-567t^{25}+\cdots-41310t^{15}+40257t^{14}+\cdots+45t^2-9t+1.
\]
\end{exe}
\begin{figure}[h]
    \centering
    \includegraphics[width=0.95\textwidth]{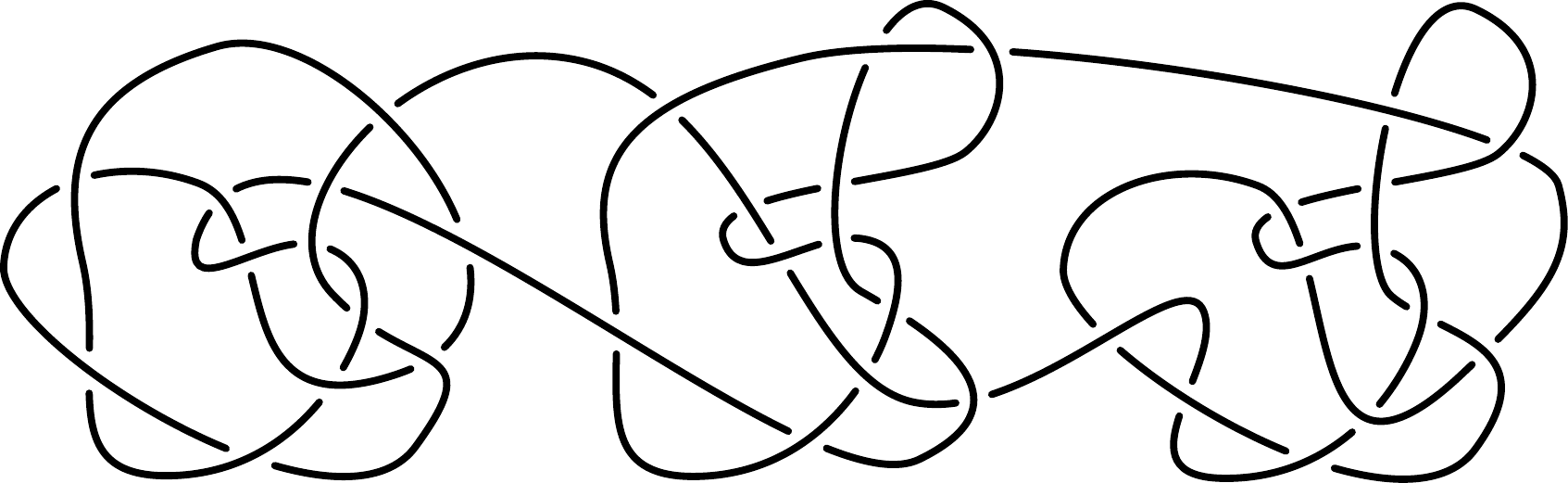}
    \caption{This figure represents the knot $\#(\gamma,3)$. As $\gamma$ (figure~\ref{fig:knot}) is $3$-trivial, this one is $9$-trivial.\\ \selectlanguage{french}\emph{Cette figure représente le noeud $\#(\gamma,3)$. Comme $\gamma$ (figure~\ref{fig:knot}) est $3$-trivial, celui-ci est $9$-trivial.}\selectlanguage{english}}
    \label{fig:knot9trivial}
\end{figure}

At the time of writing, $n$-trivial knots are only known for $2^k$ or $3^k$ and $k\geq 1$. It remains an open problem to extend this result to other moduli $n$. The naive approach described in the next proposition shows that we cannot obtain a composite module directly by connected sum.
\begin{pro}
Let $K_1$ and $K_2$ be $m_1$-trivial and $m_2$-trivial knots respectively where $m_1\neq m_2$. Then for all $n_1,n_2\geq1$, the connected sum of $\#\left(K_1,n_1\right)$ and $\#\left(K_2,n_2\right)$ is not $m_1m_2$-trivial.
\end{pro}
\begin{proof}
See proposition~\ref{pro:nmproduct}.
\end{proof}

The best composition we can obtain this way is the greatest common divisor of the moduli involved.
\begin{cor}
For $K_1, K_2,\dots, K_n$ knots which are $m_1$-trivial, $\dots$, $m_n$-trivial respectively and $k_1, k_2,\dots, k_n\geq 1$, we have:
\[
V\left(\connectsum_{i=1}^n \#\left(K_i,m_i^{k_i-1}\right)\right)\equiv 1 \left[\gcd\left(m_1^{k_1},\dots,m_n^{k_n}\right)\right]
\]
\end{cor}
\section{Arithmetic properties}\label{sec:arithmeticproperties}

To reach theorem~\ref{thm:nktrivial}, we used some arithmetic properties, in particular lemma~\ref{lem:powerpolynomialmodular}. The aim of this section is to prove this result.
\begin{nt}
For $p$ a prime number, we denote by $v_p(n)$ the $p$-adic valuation of an integer $n$.
\end{nt}

With this notation, the prime factor decomposition of an integer $n$ is $n=\prod_{i=0}^kp_i^{v_{p_i}(n)}$. We start by showing the following:
\begin{pro}\label{pro:pvaluation}
For $p$ a prime number and $k$ a positive integer, the $p$-valuation of $k!$ is smaller than or equal to $k-1$. 
\end{pro}
\begin{proof}
We denote by $s_p(k)$ the sum of the digits of $k$ written in the base-$p$ expansion. The alternate form of Legendre's formula \cite[theorem~1.2.3 p.~6]{boros2004irresistible} (see \cite[XVI p.~8]{legendre1808essai} for the original) immediately gives:
\[
v_p(k!) = \frac{k-s_p(k)}{p-1}
\]
It remains to establish the desired bound. As $k\geq 1$ we have $s_p(k)\geq1$ and as $p>1$ the expected result is obvious.
\end{proof}

The previous proposition allow us to establish the following divisibility result involving binomial coefficients at powers of $n$:
\begin{pro}\label{pro:binome}
For $n,k\geq2$ integers and $i\in\intent{1,k-1}$, we have that $n^{k-i}$ divide $\binom{n^{k-1}}{i}$.
\end{pro}
\begin{proof}
For $p$ a prime number dividing $n$, we study the $p$-valuation of $\frac{n^{k-1}}{\gcd\left(n^{k-1},i!\right)}$. By proposition~\ref{pro:pvaluation}, we know that $v_p(i!)\leq i-1\leq v_p(n)(i-1)$ as $v_p(n)\geq1$, thus:
\[
v_p\left(\frac{n^{k-1}}{\gcd\left(n^{k-1},i!\right)}\right)=v_p(n)(k-1)-v_p\left(\gcd\left(n^{k-1},i!\right)\right)\geq v_p(n)(k-1)-v_p(i!)\geq v_p(n)(k-i)
\]
We conclude that $n^{k-i}$ divides $\frac{n^{k-1}}{\gcd\left(n^{k-1},i!\right)}$, so it divides $\binom{n^{k-1}}{i}$ too.
\end{proof}

Proposition~\ref{pro:binome} has an interesting consequence on powers of specific polynomials. Let us prove now lemma~\ref{lem:powerpolynomialmodular}:
\begin{proof}[Proof of lemma~\ref{lem:powerpolynomialmodular}]
The case $n\leq 1$ is trivial. Assuming $n\geq 2$, we develop the product:
\[
P^{n^{k-1}} = (1+nQ)^{n^{k-1}} = \sum_{i=0}^{n^{k-1}}\left[\binom{n^{k-1}}{i}(nQ)^{i}\right] = 1 + \sum_{i=1}^{k-1}\left[\binom{n^{k-1}}{i}(nQ)^{i}\right] + n^kR_0
\]
with $R_0$ a remainder polynomial. The only thing left to do is to take enough power of $n$ from the combinatorial coefficient to have a factor $n^k$ appear. However, by proposition~\ref{pro:binome}, we know that $\binom{n^{k-1}}{i}$ is divisible by $n^{k-i}$ for~$i$ in $\intent{1,k-1}$, so:
\[
    1 + \sum_{i=1}^{k-1}\left[\binom{n^{k-1}}{i}(nQ)^{j}\right] + n^kR_0 = 1 + \sum_{i=1}^{k-1}\left[n^kR_i\right] + n^kR_0 \equiv 1 \left[n^k\right]
\]
where $R_i$ are some polynomials.
\end{proof}

If we take for example a polynomial of the form $P=1+3Q$, a direct computation gives $P^3=1+9Q+9Q^2+27Q^3$ which is congruent to $1$ modulo $9$.
\begin{rem}
We can generalize lemma~\ref{lem:powerpolynomialmodular} for any ring $R$. In fact for any $A=I+nB$ living in $R$, since the neutral element $I$ commutes with all elements, the binomial expansion works even if this ring is not commutative.
\end{rem}

The following proposition explains why we can't generalize directly the proof of theorem~\ref{thm:nktrivial}.
\begin{pro}\label{pro:nmproduct}
Let $P_1$ and $P_2$ be two polynomials of the form $P_1=1+nQ_1$, $P_2=1+mQ_2$ where $Q_1$ and $Q_2$ are some polynomials with at least two coprime coefficients and $n$, $m$ are different integers. Then:
\[\forall a,b\geq0, P_1^aP_2^b\not\equiv1 [nm]\]
\end{pro}
\begin{proof}
We expand using the binomial formula:
\begin{align*}
    P_1^aP_2^b &=\left(\sum_{i=0}^a\binom{a}{i}(nQ_1)^i\right)\left(\sum_{j=0}^b\binom{b}{j}(mQ_2)^b\right)\\
    &= 1 + \underbrace{\sum_{i=1}^a\binom{a}{i}(nQ_1)^i}_{=n^aQ_1^a+R_1} + \underbrace{\sum_{j=1}^b\binom{b}{j}(nQ_2)^j}_{=m^bQ_2^b+R_2} + \underbrace{\sum_{i=1}^a\sum_{j=1}^b\binom{a}{i}\binom{b}{j}(nQ_1)^i(mQ_2)^j}_{=nmR}\\
    &= 1 + n^aQ_1^a + m^bQ_2^b + R_1 + R_2 + nmR
\end{align*}
Here $R_1$, $R_2$ and $R$ are some polynomials. We clearly see that the terms $n^aQ_1^a$ and $m^bQ_2^b$ are not $0$ modulo $nm$, hence $P_1^aP_2^b$ is not equivalent to $1$ modulo $nm$.
\end{proof}
\section{Further work}\label{sec:furtherwork}

The main consequence of theorem~\ref{thm:nktrivial} is the existence for all $k$ of $2^k$-trivial and $3^k$-trivial non-prime knots. However, we do not have any information on other moduli. It will be interesting to find a $6$-trivial one, being $2$-trivial and $3$-trivial at the same time it may help to determine if the $n$-trivial property is multiplicative.

Although this property is defined on the Jones polynomial, we can imagine a similar definition on the Kauffman bracket. This might yield stronger results, surely linked with the one on the Jones polynomial. This approach was already used in \cite{eliahou2017remarkable} but with algebraic tangles only.

Another problem is determining the minimal number of crossing needed for a knot to be $n$-trivial, and also whether these "minimal" $n$-trivial knots are prime when $n$ is composite.

\bibliographystyle{plain}
\bibliography{Bibliography/references}

\begin{tabular}{cp{0.8\linewidth}}
\faUser & Guillaume Pagel\\
\faAt & \href{mailto:guillaume.pagel@univ-littoral.fr}{guillaume.pagel@univ-littoral.fr}\\
\faEnvelope & Univ. Littoral Côte d'Opale, UR 2597, LMPA, Laboratoire de Mathématiques Pures et Appliquées Joseph Liouville, F-62100 Calais, France\\
\end{tabular}

\end{document}